\documentclass[11pt]{amsart}

\usepackage[T1]{fontenc}
\usepackage{lmodern}
\usepackage{microtype}
\usepackage{amsmath,amssymb,amsthm,mathtools}
\usepackage{enumitem}
\usepackage{booktabs}
\usepackage{array}
\usepackage{longtable}
\usepackage{tikz-cd}
\usepackage[colorlinks=true,linkcolor=blue!45!black,citecolor=blue!45!black,urlcolor=blue!45!black]{hyperref}
\usepackage[nameinlink,capitalise,noabbrev]{cleveref}

\setlist[enumerate,1]{label=\textup{(\arabic*)},leftmargin=2.3em}
\setlist[itemize]{leftmargin=1.8em}
\allowdisplaybreaks

\newtheorem{theorem}{Theorem}[section]
\newtheorem{proposition}[theorem]{Proposition}
\newtheorem{corollary}[theorem]{Corollary}
\newtheorem{lemma}[theorem]{Lemma}

\theoremstyle{definition}
\newtheorem{definition}[theorem]{Definition}
\newtheorem{example}[theorem]{Example}
\theoremstyle{remark}
\newtheorem{remark}[theorem]{Remark}

\newcommand{\K}{K}
\newcommand{\Z}{\mathbb Z}
\newcommand{\G}{\mathcal G}

\newcommand{\units}{\G^{(0)}}

\newcommand{\supp}{\operatorname{supp}}

\newcommand{\fin}{\mathrm{fin}}
\newcommand{\one}{\mathbf 1}
\newcommand{\cO}{\mathcal O}

\newcommand{\cU}{\mathcal U}
\newcommand{\cP}{\mathcal P}

\title[Baer-Type Ring Characterizations of Steinberg Algebras]{Baer-Type Ring Characterizations of Steinberg Algebras, with Applications to Leavitt Path Algebras}

\author{Morteza Ahmadi}
\address{Department of Pure Mathematics, Faculty of Mathematical Sciences\\
	Tarbiat Modares University, P.O.Box:14115-134, Tehran, Iran}
\email{morteza.ahmadi23@gmail.com\,\ mortezy.ahmadi@modares.ac.ir}

\subjclass[2020]{Primary 16S88, 22A22; Secondary 16W10, 16W50, 16D70, 16S10}
\keywords{Steinberg algebra, ample groupoid, Baer ring, Rickart ring, involution, graded ring, Leavitt path algebra, graph groupoid}

\hypersetup{
  pdftitle={Baer-Type Ring Characterizations of Steinberg Algebras, with Applications to Leavitt Path Algebras},
  pdfauthor={Morteza Ahmadi},
  pdfsubject={Baer, Rickart, graded, and involutive conditions for Steinberg algebras and Leavitt path algebras},
  pdfkeywords={Steinberg algebra, ample groupoid, Baer ring, Rickart ring, Leavitt path algebra}
}

\begin{document}

\begin{abstract}
We extend the theory of Baer, Rickart, and their local, graded, and  $*$-analogues for Steinberg algebras. 
First, we prove that positive definiteness passes from the coefficient field to the Steinberg algebra, and that this, together with graded von Neumann regularity, characterized by Steinberg and van Wyk, 
yields graded locally Rickart $*$-Steinberg algebras.
Then we show that local Baer behavior forces every compact open subset of the unit space to be extremally disconnected. Consequently, when compact open subspaces of the unit space are metrizable, local Baer behavior forces the groupoid to be discrete. 
In the discrete case, we use the orbit-by-orbit decomposition into finitary matrix algebras over isotropy group algebras to convert annihilator conditions into explicit matrix conditions.
For groupoids whose isotropy is trivial or infinite cyclic, these ingredients yield complete local and unital characterizations.  In particular, local Baer and graded local Baer conditions are equivalent to discreteness, while the ungraded local Baer $*$-condition additionally requires that every orbit with nontrivial isotropy be a singleton. 
We then apply the theory to boundary path groupoids.  Using the Steinberg algebra model, we 
obtain the Rickart, Baer,  and Baer $*$ characterizations of Leavitt path algebras, including the local and graded variants.  
\end{abstract}

\maketitle

\section{Introduction}
Steinberg algebras provide an algebraic interface between ample groupoids, inverse semigroup algebras, graph algebras, and graded ring theory \cite{Steinberg2010}. Their compact open bisections encode the topology and dynamics of a \linebreak groupoid, while convolution converts those data into a ring with local units.  This point of view has proved especially effective for Leavitt path algebras: the boundary path groupoid of a directed graph has a Steinberg algebra naturally isomorphic, as a graded algebra with involution, to the corresponding Leavitt path algebra \cite{ClarkFarthingSimsTomforde2014,Rigby2020}.  Graded Steinberg algebras and their module categories were developed in \cite{AraHazratLiSims2018}; recently, regularity and graded regularity have been characterized in terms of ample groupoids \cite{SteinbergVanWyk2026}.

For a ring $R$, the right annihilator of a subset $X\subseteq R$ is
\[
  r_R(X)=\{a\in R~|~xa=0\text{ for all }x\in X\}.
\]
A unital ring is right Rickart when the annihilator of every single element is generated by an idempotent, and it is Baer when the same is true for every subset.  Replacing idempotents by projections gives Rickart $*$-rings and Baer $*$-rings.  These properties arose in the algebraic study of operator rings (see \cite{Berberian1972,Berberian1988}).  Hazrat and Va\v{s} obtained graph-theoretic characterizations of these conditions for Leavitt path algebras, including local and graded versions \cite{HazratVas2018}.  
 They show, in particular, that a Leavitt path algebra is Baer exactly for a finite no-exit graph and, over a positive definite $*$-field, is Baer~$*$ exactly for a finite disjoint union of finite acyclic graphs and isolated loops.

The purpose of this paper is to move the annihilator problem one level upward: first analyze it for Steinberg algebras, and then recover the graph results by specialization.  

Our first result concerns positivity.  If $\G$ is a nonempty Hausdorff ample groupoid and $\K$ is a field with involution, then $A_{\K}(\G)$ is positive definite if and only if $K$ is positive definite.
Combined with graded regularity, it yields a characterization of graded locally Rickart  $*$-Steinberg algebras.

Our next step uses diagonal characteristic functions.  Let $U\subseteq\units$ be compact open and let $\{V_i\}$ be compact open subsets of $U$.  In the corner $A_{\K}(\G|_U)$, the right annihilator of the functions $\one_{V_i}$ consists precisely of functions whose range support lies in
\[
U\setminus\overline{\bigcup_iV_i}.
\]
If this annihilator is generated by an idempotent, the latter set must itself be compact open.  Thus local Baer behavior forces each compact open subspace of the unit space to be extremally disconnected.  If these subspaces are metrizable, they must be finite, so the entire groupoid is discrete.  The metrizability hypothesis is essential (see \Cref{rem:metric-essential}). 

For discrete groupoids, the remaining problem is algebraic.  Choosing a representative $u$ in each orbit $\cO$ and arrows from $u$ to the other units gives
\[
 A_{\K}(\G)
 \cong
 \bigoplus_{\cO\in \units/\G}
 M^{\fin}_{\cO}\bigl(\K\G_u^u\bigr).
\]
The isomorphism respects the canonical involution and, in the presence of a cocycle, the appropriate shifted grading.  Hence local annihilator conditions reduce to finite matrix corners over isotropy group algebras.  When every isotropy group is trivial or infinite cyclic, the only coefficient blocks are $\K$ and $\K[t,t^{-1}]$.  Their Baer behavior is benign, but $M_n(\K[t,t^{-1}])$ is not Rickart~$*$ for $n>1$.  This precisely detects the extra orbit condition required for ungraded local Baer $*$.

The graph application requires no countability assumption.  We prove that every infinite compact open subset of a boundary path space contains a nontrivial convergent sequence.  Since compact extremally disconnected spaces contain no such sequence, the local Baer condition forces the boundary path space to be discrete.  The latter is equivalent to the familiar graph condition: the graph is row-finite and no-exit, and every infinite path eventually becomes an infinite sink or periodic cycle.  The orbit-isotropy decomposition then recovers the matrix and Laurent polynomial blocks and yields all corresponding Leavitt path algebra results.

The paper is organized as follows.  \Cref{sec:prelim} reviews annihilator conditions and Steinberg algebras.  \Cref{sec:positive} establishes positive definiteness and Rickart consequences.  \Cref{sec:topological} develops the extremal-disconnectedness obstruction.  \Cref{sec:discrete} gives the orbit-matrix decomposition for discrete groupoids.  \Cref{sec:cyclic} treats cyclic isotropy.  \Cref{sec:graphs} specializes to graph groupoids and Leavitt path algebras.  Examples 
appear in \Cref{sec:examples}.

\section{Preliminaries}\label{sec:prelim}

\subsection{Rickart and Baer conditions: local, graded, and $*$-versions}

Throughout the paper, all rings are associative, but they need not be unital unless explicitly stated.  For $X\subseteq R$, define
\begin{align*}
r_R(X)&=\{a\in R~|~xa=0\text{ for all }x\in X\},\\
\ell_R(X)&=\{a\in R~|~ax=0\text{ for all }x\in X\}.
\end{align*}

A ring $R$ is \emph{right Rickart} if, for each $x\in R$, there is an idempotent $e=e^2$ such that $r_R(x)=eR$.  Left Rickart is defined analogously, and Rickart means both.  A unital ring is \emph{Baer} if, for each subset $X\subseteq R$, the right ideal $r_R(X)$ is generated by an idempotent.  The Baer condition is left-right symmetric.

A $*$-ring is a ring equipped with an involutive anti-automorphism $x\mapsto x^*$.  A projection is an element $p=p^2=p^*$.  A $*$-ring is Rickart $*$ (resp., Baer $*$), when the idempotent generators in the corresponding definitions can always be chosen to be projections.   

Rings with Rickart or Baer conditions, or their $*$-versions, are unital, because the generator of the annihilator of $0$ is an identity.

Let $\Gamma$ be an abelian group and $R=\bigoplus_{\gamma\in\Gamma}R_\gamma$ a $\Gamma$-graded ring.  The ring is a graded $*$-ring if $R_\gamma^*\subseteq R_{-\gamma}$.  It is graded right Rickart if the right annihilator of every homogeneous element is generated by a homogeneous idempotent, and graded Baer if the right annihilator of every set of homogeneous elements is generated by a homogeneous idempotent.  The graded $*$-versions use homogeneous projections.  These definitions reduce to the ordinary ones for the trivial grading.

Recall from \cite{HazratVas2018} that a possibly nonunital ring $R$ is \emph{locally right Rickart} if  $eRe$ is right Rickart for every idempotent $e\in R$. Locally Rickart and locally Baer are defined analogously. If  $R$ is a $*$-ring, it is locally Rickart $*$ (resp., locally Baer $*$) when $pRp$ is  Rickart $*$ (resp.,  Baer $*$) for every projection $p\in R$. In the graded setting, one takes corners by homogeneous idempotents, or by homogeneous projections for the $*$-conditions.

The conditions above pass to corners of the appropriate type.  Thus, if $R$ has a specified set $\cU$ of local units, it is enough to check the condition on $uRu$ for $u\in\cU$.  We use repeatedly the following facts. 

\begin{lemma}[\cite{HazratVas2018}, Sections 1 and 2]\label{lem:standard-ring-facts}
Let $R$ be a ring.
\begin{enumerate}
\item\label{lem2.3_1} If $R$ is locally unital and right semihereditary, then $R$ is locally right Rickart.
\item\label{lem2.3_2} If $R$ is a graded regular graded $*$-ring with homogeneous local $*$-units and graded proper involution, then $R$ is graded locally Rickart $*$.
\item\label{lem2.3_3} A Rickart ring is graded Rickart for every grading for which it is a graded ring, and the corresponding local statement holds.  A Baer ring is graded Baer.  The analogous implications hold for the $*$-versions.
\item\label{lem2.3_4} The properties considered above pass to corners and finite direct products.
\end{enumerate}
\end{lemma}





An involution is \emph{positive definite} if, for every finite family $x_1,\ldots,x_n$,
\[
  \sum_{i=1}^n x_ix_i^*=0
  \quad\Longrightarrow\quad
  x_1=\cdots=x_n=0.
\]
It is \emph{proper} if $xx^*=0$ implies $x=0$.  Positive definiteness implies properness.  In a graded $*$-ring, graded properness means the same implication for homogeneous $x$.

\subsection{Ample groupoids and Steinberg algebras}

Let $\G$ be a Hausdorff ample groupoid.  Thus $\G$ is an etale groupoid, its unit space $\units$ is locally compact Hausdorff and zero-dimensional, and the compact open bisections form a basis for the topology.  We write $\mathbf{r,s}:\G\to\units$ for range and source.  For $U\subseteq\units$, the reduction is
\[
   \G|_U=\mathbf{r}^{-1}(U)\cap \mathbf{s}^{-1}(U).
\]
The isotropy group at $u\in\units$ is $\G_u^u=\mathbf{r}^{-1}(u)\cap \mathbf{s}^{-1}(u)$, and the orbit of $u$ is $\G u=\mathbf{r}(\mathbf{s}^{-1}(u))$.

Let $\K$ be a field.  The Steinberg algebra $A_{\K}(\G)$ is the $\K$-vector space of compactly supported, locally constant functions $f:\G\to\K$, with convolution
\[
 (f*g)(\gamma)=\sum_{\alpha\beta=\gamma}f(\alpha)g(\beta).
\]
The sum is finite.  The characteristic functions $\one_B$ of compact open bisections span the algebra, and
\[
 \one_B*\one_C=\one_{BC}.
\]
If $U\subseteq\units$ is compact open, then $\one_U$ is an idempotent, the family of all such functions is a set of local units, and
\begin{equation}\label{eq:corner-reduction}
  \one_U A_{\K}(\G)\one_U\cong A_{\K}(\G|_U).
\end{equation}
The algebra is unital exactly when $\units$ is compact, in which case its identity is $\one_{\units}$.

Suppose that $\K$ is a field with involution.  Then $A_{\K}(\G)$ is a $*$-algebra under
\begin{equation}\label{eq:steinberg-star}
 f^*(\gamma)=\overline{f(\gamma^{-1})}.
\end{equation}
Every $\one_U$ with $U\subseteq\units$ compact open is a projection.

Let $c:\G\to\Gamma$ be a continuous cocycle into a discrete abelian group.  Then
\[
 A_{\K}(\G)_\gamma
 =\{f\in A_{\K}(\G)~|~\supp(f)\subseteq c^{-1}(\gamma)\}
\]
gives a $\Gamma$-grading, and \eqref{eq:steinberg-star} makes it a graded $*$-algebra.  This construction and its representation theory are discussed in \cite{AraHazratLiSims2018}.

\section{Positive definiteness and graded Rickart $*$-condition}\label{sec:positive}

The following result is independent of freeness, effectiveness, isotropy, and countability.

\begin{proposition}\label{thm:positive definite}
Let $\G$ be a nonempty Hausdorff ample groupoid and let $\K$ be a field with involution.  Then the  involution on $A_{\K}(\G)$ is positive definite if and only if the involution on $\K$ is positive definite.
\end{proposition}

\begin{proof}
Assume first that $\K$ is positive definite.  Let $f_1,\ldots,f_n\in A_{\K}(\G)$ satisfy
\[
 \sum_{i=1}^n f_i*f_i^*=0.
\]
For a unit $u\in\units$, the convolution formula gives
\begin{align*}
 0
 &=\sum_{i=1}^n(f_i*f_i^*)(u)\\
 &=\sum_{i=1}^n\sum_{\alpha\beta=u}f_i(\alpha)f_i^*(\beta)\\
 &=\sum_{i=1}^n\sum_{\gamma\in r^{-1}(u)}
     f_i(\gamma)\overline{f_i(\gamma)}.
\end{align*}
Since the $f_i$ have compact support and $\G$ is etale, only finitely many summands are nonzero.   Positive definiteness of $\K$ implies $f_i(\gamma)=0$ for every $i$ and every $\gamma\in \mathbf{r}^{-1}(u)$.  Because $u\in\units$ was arbitrary, $f _i=0$ for all $i$.

Conversely, assume that $A_{\K}(\G)$ is positive definite.  Choose a nonempty compact open set $U\subseteq\units$.  If $k_1,\ldots,k_n\in\K$ satisfy $\sum_i k_i\overline{k_i}=0$, then
\[
  \sum_{i=1}^n(k_i\one_U)(k_i\one_U)^*
  =\left(\sum_{i=1}^n k_i\overline{k_i}\right)\one_U=0.
\]
Positive definiteness of $A_{\K}(\G)$ yields $k_i\one_U=0$, and hence $k_i=0$ for every~$i$.
\end{proof}

\begin{corollary}\label{cor:graded-rickart-star}
Let $c:\G\to\Gamma$ be a continuous cocycle.  Suppose that $\K$ is positive definite and $A_{\K}(\G)$ is graded von Neumann regular for the cocycle grading (i.e., the $\Gamma$-grading induced by $c$).
 Then $A_{\K}(\G)$ is graded locally Rickart $*$.  If $\units$ is compact, it is a graded Rickart $*$-ring.
\end{corollary}

\begin{proof}
By \Cref{thm:positive definite}, the canonical involution is positive definite, hence graded proper.  The compact open unit projections are homogeneous local $*$-units.  Apply \Cref{lem:standard-ring-facts}\ref{lem2.3_2}.  Compactness of $\units$ supplies an identity.
\end{proof}

Steinberg and van Wyk proved that, 
graded regularity is equivalent to regularity of the neutral component; they then characterized the latter by approximate quasi-compactness and isotropy conditions \cite{SteinbergVanWyk2026}.  The following combines their theorem with \Cref{cor:graded-rickart-star}.

\begin{corollary}\label{cor:svw-rickart}
Let $c:\G\to\Gamma$ be a continuous cocycle, let $\G_0=c^{-1}(0)$, and let $\K$ be a positive definite field.  If $\G_0$ is a directed union of quasi-compact open subgroupoids and the order of every finite subgroup of every isotropy group of $\G_0$ is invertible in $\K$, then $A_{\K}(\G)$ is graded locally Rickart $*$.
\end{corollary}

\begin{proof}
A field is von Neumann regular.  The hypotheses imply that $A_{\K}(\G_0)$ is regular, and hence $A_{\K}(\G)$ is graded regular by \cite[Theorems 4.2 and 4.3]{SteinbergVanWyk2026}.  Apply \Cref{cor:graded-rickart-star}.
\end{proof}

\section{Locally Baer Steinberg algebras have Stonean unit corners}\label{sec:topological}

This section extracts a topological consequence of the Baer property using only characteristic functions of compact open subsets of the unit space.

\begin{lemma}\label{lem:diagonal-annihilator}
Let $U\subseteq\units$ be compact open, let $B=A_{\K}(\G|_U)$, and let $\{V_i~|~i\in I\}$ be a family of compact open subsets of $U$.  Put
\[
  O=\bigcup_{i\in I}V_i,
  \qquad
  W=U\setminus\overline O,
\]
where the closure is taken in $U$.  Then
\begin{equation}\label{eq:diagonal-annihilator}
 r_B\bigl(\{\one_{V_i}~|~i\in I\}\bigr)
 =\{f\in B~|~\supp(f)\subseteq \mathbf{r}^{-1}(W)\}.
\end{equation}
\end{lemma}

\begin{proof}
For $V\subseteq U$ compact open, $f\in B$, and $\gamma\in\G|_U$ we have 
\begin{equation}\label{eq:left-diagonal}
  (\one_V*f)(\gamma)=\one_V(\mathbf{r}(\gamma))f(\gamma).
\end{equation}
Consequently, $f$ lies in the annihilator on the left-hand side of \eqref{eq:diagonal-annihilator} if and only if $f(\gamma)=0$ whenever $\mathbf{r}(\gamma)\in O$.

Suppose $\mathbf{r}(\gamma)\in\overline O$ and $f(\gamma)\neq0$.  Since $f$ is locally constant and compact open bisections form a basis, there is a compact open bisection $C$ containing $\gamma$ on which $f$ is a fixed nonzero scalar.  The open set $\mathbf{r}(C)$ meets $O$, so there is $\eta\in C$ with $\mathbf{r}(\eta)\in O$, contradicting the previous paragraph.  Hence $\supp(f)\subseteq \mathbf{r}^{-1}(W)$.  The converse follows immediately from \eqref{eq:left-diagonal}.
\end{proof}

Recall that a compact Hausdorff space $X$ is \emph{extremally disconnected}, or \emph{Stonean}, if the closure of every open subset of $X$ is open.

\begin{theorem}\label{thm:local-baer-stonean}
Let $\G$ be a Hausdorff ample groupoid.  If $A_{\K}(\G)$ is locally Baer, then every compact open subspace of $\units$ is extremally disconnected.

The same conclusion holds if $A_{\K}(\G)$ is graded locally Baer for a cocycle grading, locally Baer $*$, or graded locally Baer $*$.
\end{theorem}

\begin{proof}
Fix compact open $U\subseteq\units$ and an open subset $O\subseteq U$.  Since compact open sets form a basis, write $O=\bigcup_{i\in I}V_i$ with each $V_i$ compact open.  By local Baer, the corner
\[
 B=\one_U A_{\K}(\G)\one_U\cong A_{\K}(\G|_U)
\]
is Baer.  Hence there is an idempotent $e\in B$ such that
\[
 r_B\bigl(\{\one_{V_i}~|~i\in I\}\bigr)=eB.
\]
Set $W=U\setminus\overline O$.  By \Cref{lem:diagonal-annihilator}, $e$ is supported in $\mathbf{r}^{-1}(W)$, and therefore
\begin{equation}\label{eq:range-one-inclusion}
 \mathbf{r}(\supp(e))\subseteq W.
\end{equation}

For the reverse inclusion, let $x\in W$.  Choose a compact open neighborhood $V$ with $x\in V\subseteq W$.  Again by \Cref{lem:diagonal-annihilator}, $\one_V\in eB$, so $\one_V=e*g$ for some $g\in B$.  Evaluating at $x$ gives
\[
 1=(e*g)(x)=\sum_{\gamma\in r^{-1}(x)}e(\gamma)g(\gamma^{-1}).
\]
At least one summand has $e(\gamma)\neq0$, and hence $x\in \mathbf{r}(\supp(e))$.  Together with \eqref{eq:range-one-inclusion}, this proves
\[
  W=\mathbf{r}(\supp(e)).
\]
The support of a locally constant compactly supported function is compact open.  Covering it by finitely many compact open bisections shows that its range is also compact open.  Thus $W$ is compact open in $U$, and
\[
 \overline O=U\setminus W
\]
is open.  Therefore $U$ is extremally disconnected.

For the graded Baer condition, the family $\{\one_{V_i}\}$ consists of homogeneous degree-zero elements, so the same argument applies.  The Baer $*$ versions imply the corresponding Baer conclusion after replacing a projection by its underlying idempotent.
\end{proof}
We will use the following standard fact. A metrizable extremally disconnected space is discrete; hence, if it is also compact, it must be finite.
\begin{lemma}\label{lem:metric-stonean-finite}
Every compact metrizable extremally disconnected space is finite.
\end{lemma}


\begin{corollary}\label{cor:local-baer-discrete}
Suppose every compact open subset of $\units$ is metrizable; this holds, in particular, when $\G$ is second countable.  If $A_{\K}(\G)$ satisfies any one of the four local Baer conditions in \Cref{thm:local-baer-stonean}, then $\G$ is discrete.
\end{corollary}

\begin{proof}
By \Cref{thm:local-baer-stonean,lem:metric-stonean-finite}, every compact open subset of $\units$ is finite.  Since such sets form a basis, $\units$ is discrete.  If $\gamma\in\G$, choose an open bisection $B$ containing $\gamma$.  The source $\mathbf{s}(B)$ is discrete, so restricting $B$ to the singleton $\{\mathbf{s}(\gamma)\}$ shows that $\{\gamma\}$ is open.  Thus $\G$ is discrete.
\end{proof}

The metrizability hypothesis in \Cref{cor:local-baer-discrete} is necessary, as the following example shows.

\begin{example}\label{rem:metric-essential}
Let $X$ be an infinite Stonean compact space, for example the Stone--\v{C}ech compactification $\beta\mathbb N$, and regard $X$ as a unit groupoid. Then $A_{\K}(X)$ is the algebra of locally constant $\K$-valued functions on $X$.  For a set $S$ of such functions, if $O$ is the union of their supports, then
\[
  r_{A_{\K}(X)}(S)=\one_{X\setminus\overline O}A_{\K}(X).
\]
Since $X$ is extremally disconnected, $X\setminus\overline O$ is clopen.  Hence $A_{\K}(X)$ is Baer although $X$ need not be discrete.  
\end{example}

\section{Discrete groupoids and orbit-matrix decompositions}\label{sec:discrete}

For a set $I$ and a ring $R$, let $M_I^{\fin}(R)$ denote the ring of $I\times I$ matrices over $R$ having only finitely many nonzero entries.  It is unital exactly when $I$ is finite.

\begin{theorem}[Orbit-matrix decomposition]\label{thm:orbit-matrix}
Let $\G$ be a discrete groupoid and let $\K$ be a field.  For each orbit $\cO\in\units/\G$, choose a base point $u_{\cO}\in\cO$ and, for each $v\in\cO$, an arrow $\tau_v:u_{\cO}\to v$, with $\tau_{u_{\cO}}=u_{\cO}$.  Put $H_{\cO}=\G_{u_{\cO}}^{u_{\cO}}$.  Then there is an algebra isomorphism
\begin{equation}\label{eq:orbit-decomp}
 \Phi:A_{\K}(\G)
 \stackrel{\cong}{\longrightarrow}
 \bigoplus_{\cO\in\units/\G}M_{\cO}^{\fin}(\K H_{\cO})
\end{equation}
determined by
\begin{equation}\label{eq:orbit-map}
 \Phi\bigl(\delta_{\tau_vh\tau_w^{-1}}\bigr)=E_{v,w}h
 \qquad(v,w\in\cO,\ h\in H_{\cO}).
\end{equation}
If $\K$ has an involution and each group algebra has the canonical involution, then $\Phi$ is a $*$-isomorphism.

If $c:\G\to\Gamma$ is a cocycle, define $\gamma_v=c(\tau_v)$ and grade $\K H_{\cO}$ by $\deg(h)=c(h)$.  With the shifted matrix grading
\[
 \deg(E_{v,w}h)=\gamma_v+c(h)-\gamma_w,
\]
$\Phi$ is a graded $*$-isomorphism.
\end{theorem}

\begin{proof}
Every arrow $\alpha:w\to v$ in the reduction $\G|_{\cO}$ has a unique expression
\[
  \alpha=\tau_vh\tau_w^{-1},
  \qquad
  h=\tau_v^{-1}\alpha\tau_w\in H_{\cO}.
\]
The point masses form a basis of the Steinberg algebra because $\G$ is discrete and every compact set is finite.  Formula \eqref{eq:orbit-map} therefore defines a linear bijection.  Uniqueness of the expression and the multiplication rule for groupoid arrows show that it preserves products.  Inversion sends
\[
 (\tau_vh\tau_w^{-1})^{-1}=\tau_wh^{-1}\tau_v^{-1}
\]
to the conjugate-transpose of $E_{v,w}h$, proving the $*$-statement.  Finally,
\[
 c(\tau_vh\tau_w^{-1})=c(\tau_v)+c(h)-c(\tau_w),
\]
which proves gradedness.
\end{proof}

\begin{theorem}\label{thm:discrete-local-criterion}
Use the notation of \Cref{thm:orbit-matrix}.  Let $\cP$ be any one of the properties Rickart, Baer, Rickart $*$, or Baer $*$, and use homogeneous versions throughout when a cocycle is fixed.  Then $A_{\K}(\G)$ is locally $\cP$ if and only if
\begin{equation}\label{eq:finite-matrix-criterion}
 M_F(\K H_{\cO})\text{ has }\cP
\end{equation}
for every orbit $\cO$ and every finite nonempty subset $F\subseteq\cO$, with the grading and involution inherited from \Cref{thm:orbit-matrix}.

If $\units$ is finite, then $A_{\K}(\G)$ is unital and has $\cP$ if and only if every full orbit block $M_{\cO}(\K H_{\cO})$ has $\cP$.
\end{theorem}

\begin{proof}
For finite $F\subseteq\cO$, the idempotent
\[
 p_F=\sum_{v\in F}\delta_v
\]
is a homogeneous projection and
\[
 p_F A_{\K}(\G)p_F\cong M_F(\K H_{\cO}).
\]
Thus local $\cP$ implies \eqref{eq:finite-matrix-criterion}.

Conversely, let $e$ be an idempotent, or a homogeneous projection in the relevant graded $*$-case.  Its support is finite.  Hence $e\in pA_{\K}(\G)p$ where $p$ is a finite sum of unit point masses, involving only finitely many orbits and finitely many units in each orbit.  By \eqref{eq:orbit-decomp}, this corner is a finite direct product of rings of the form \eqref{eq:finite-matrix-criterion}.  The property $\cP$ passes to finite products and corners, so $eA_{\K}(\G)e$ has $\cP$.  This proves the local assertion.  The unital statement follows directly from the finite direct-product decomposition.
\end{proof}

Combining the topological obstruction with the matrix criterion gives the following. 

\begin{corollary}\label{cor:general-reduction}
Assume every compact open subset of $\units$ is metrizable.  Then $A_{\K}(\G)$ is locally Baer if and only if
\begin{enumerate}
\item $\G$ is discrete; and
\item $M_F(\K\G_u^u)$ is Baer for every $u\in\units$ and every finite subset $F$ of $\G u$. 
\end{enumerate}
Analogous statements hold for graded local Baer, local Baer $*$, and graded local Baer $*$, with the corresponding matrix property.
\end{corollary}

\begin{proof}
Necessity of discreteness is \Cref{cor:local-baer-discrete}; the rest is \Cref{thm:discrete-local-criterion}.
\end{proof}

\section{Cyclic-isotropy Steinberg algebras}\label{sec:cyclic}

We now obtain explicit characterizations for the isotropy pattern occurring in graph groupoids.

\begin{definition}
A groupoid $\G$ has \emph{cyclic isotropy} if every isotropy group is either trivial or infinite cyclic.  A cocycle $c:\G\to\Z$ is \emph{isotropy-faithful} if its restriction to each isotropy group is injective.
\end{definition}


\begin{lemma}[\cite{HazratVas2018}, Lemma 14]\label{lem:matrix-laurent}
Let $\K$ be a field and let $n\geq1$.
\begin{enumerate}
\item\label{lem6.2_1} The rings $M_n(\K)$ and $M_n(\K[t,t^{-1}])$ are Baer.
\item\label{lem6.2_2} Give $\K$ the trivial $\Z$-grading and give $\K[t^d,t^{-d}]$ its natural grading, where $d\geq1$.  Every shifted matrix grading on
\[
 M_n(\K)
 \quad\text{and}\quad
 M_n(\K[t^d,t^{-d}])
\]
is graded regular.  If $\K$ is positive definite, these rings are graded Baer~$*$.
\item\label{lem6.2_3} If $\K$ is positive definite, $M_n(\K)$  and $\K[t,t^{-1}]$ are Baer~$*$.
\item\label{lem6.2_4} The ring $M_n(\K[t,t^{-1}])$ is Rickart $*$ if and only if $n=1$.
\end{enumerate}
\end{lemma}

\begin{theorem}\label{thm:cyclic-local}
Let $\G$ be a Hausdorff ample groupoid with cyclic isotropy, and assume every compact open subset of $\units$ is metrizable.
\begin{enumerate}
\item\label{theo6.3_1}  The following are equivalent:
\begin{enumerate}[label=\textup{(\roman*)}]
\item$\G$ is discrete;
\item $A_{\K}(\G)$ is locally Baer;
\item $A_{\K}(\G)$ is graded locally Baer for any cocycle grading.
\end{enumerate}
\item\label{theo6.3_2}  Suppose $\K$ is positive definite and $c:\G\to\Z$ is isotropy-faithful.  Then the conditions in \ref{theo6.3_1}  are also equivalent to $A_{\K}(\G)$ being graded locally Baer $*$ 
for the cocycle grading.
\item\label{theo6.3_3}  Suppose $\K$ is positive definite.  Then $A_{\K}(\G)$ is locally Baer $*$ if and only if $\G$ is discrete and every orbit containing a unit with nontrivial isotropy is a singleton.
\end{enumerate}
\end{theorem}

\begin{proof}
Any of the algebraic conditions in \ref{theo6.3_1} or \ref{theo6.3_2} forces discreteness by \Cref{cor:local-baer-discrete}.  Conversely, assume $\G$ is discrete.  By \Cref{thm:orbit-matrix}, every finite local-unit corner is a finite direct product of rings of the form
\[
 M_n(\K)
 \quad\text{or}\quad
 M_n(\K[t,t^{-1}]).
\]
These rings are Baer by \Cref{lem:matrix-laurent}\ref{lem6.2_1}, so \Cref{thm:discrete-local-criterion} gives local Baer.  An ordinary Baer ring is graded Baer, giving \ref{theo6.3_1}.

For \ref{theo6.3_2}, isotropy-faithfulness identifies the graded isotropy algebra of a nontrivial isotropy group with $\K[t^d,t^{-d}]$ for some $d\geq1$.  The orbit decomposition carries the grading to a shifted matrix grading.  Apply \Cref{lem:matrix-laurent}\ref{lem6.2_2} to every finite block and then \Cref{thm:discrete-local-criterion}.

For \ref{theo6.3_3}, assume first that $A_{\K}(\G)$ is locally Baer $*$.  It is locally Baer, so $\G$ is discrete.  Let $\cO$ be an orbit with nontrivial isotropy.  If $\cO$ contained two distinct units, a two-by-two local-unit corner would be $*$-isomorphic to $M_2(\K[t,t^{-1}])$, contradicting \Cref{lem:matrix-laurent}\ref{lem6.2_4}. Thus $\cO$ is a singleton.

Conversely, assume the stated groupoid conditions.  A finite matrix corner over a trivial-isotropy orbit is $M_n(\K)$ and is Baer $*$ by \Cref{lem:matrix-laurent}\ref{lem6.2_3}.  A nontrivial-isotropy orbit is a singleton, so its only orbit block is $\K[t,t^{-1}]$, also Baer $*$.  The discrete local criterion completes the proof.
\end{proof}

\begin{corollary}\label{cor:cyclic-unital}
Under the hypotheses of \Cref{thm:cyclic-local}, the following are equivalent:
\begin{enumerate}
\item $\units$ is finite;
\item $A_{\K}(\G)$ is Baer;
\item $A_{\K}(\G)$ is graded Baer for any cocycle grading.
\end{enumerate}
If $\K$ is positive definite and $c:\G\to\Z$ is isotropy-faithful, these are also equivalent to $A_{\K}(\G)$ being graded Baer $*$.

Moreover, for positive definite $\K$, the algebra $A_{\K}(\G)$ is Baer $*$ if and only if $\units$ is finite and every orbit with nontrivial isotropy is a singleton.
\end{corollary}

\begin{proof}
A Baer, graded Baer, or graded Baer $*$ ring is unital; hence $\units$ is compact.  The corresponding local condition forces $\G$ to be discrete by \Cref{thm:cyclic-local}, so compactness of $\units$ makes it finite.  Conversely, if $\units$ is finite, etaleness implies that $\G$ is discrete.  The orbit decomposition is then a finite product of finite matrix blocks, and \Cref{lem:matrix-laurent} applies.  The final assertion follows from \Cref{thm:cyclic-local}\ref{theo6.3_3}.
\end{proof}

\begin{remark}\label{rem:two-star-worlds}
The distinction between graded and ungraded Baer $*$ is now visible at the groupoid level.  A transitive discrete groupoid with isotropy $\Z$ and an orbit of size $n>1$ contributes the block
\[
 M_n(\K[t,t^{-1}]).
\]
With the isotropy-faithful grading this is graded Baer $*$, but without the grading it is not even Rickart $*$.  The grading separates Laurent degrees and removes the obstruction created by the idempotent used in \Cref{lem:matrix-laurent}\ref{lem6.2_4}.
\end{remark}

\section{boundary path groupoids and Leavitt path algebras}\label{sec:graphs}

We now specialize the preceding results to directed graphs.  This section recovers the known Leavitt path algebra characterizations and provides a conceptual explanation of them.

\subsection{The graph groupoid model and Leavitt path algebras}

Let $E=(E^0,E^1,\mathbf{r,s})$ be a directed graph. Vertices are regarded as paths of length zero.  For a finite path $\alpha=e_1\cdots e_n$, write
\[
|\alpha|=n,\qquad
\alpha^*=e_n^*\cdots e_1^*.
\]
 A vertex is \emph{regular} if it emits a finite nonzero number of edges, and is \emph{singular} if it is a sink or an infinite emitter.
 
 Let $\K$ be a field with involution.  The Leavitt path algebra $L_{\K}(E)$ is generated by
 \[
 \{v~|~v\in E^0\}\cup\{e,e^*~|~e\in E^1\}
 \]
 subject to the usual Cuntz--Krieger relations.  It carries the canonical involution extending that of $\K$ and interchanging $e$ and $e^*$, and the canonical $\Z$-grading
 \[
 \deg v=0,\qquad \deg e=1,\qquad \deg e^*=-1.
 \]
 Finite sums of vertices are local projection units.  We use \cite{AbramsArandaPino2005,AbramsAraSilesMolina2017} for standard facts about arbitrary-graph Leavitt path algebras.
 
  The boundary path space is
\[
 \partial E
 =E^\infty\cup\{\mu\in E^*~|~\mathbf{r}(\mu)\text{ is singular}\}.
\]
For a finite path $\mu$ and a finite set $F\subseteq \mathbf{s}^{-1}(\mathbf{r}(\mu))$, put
\[
 Z(\mu\setminus F)
 =Z(\mu)\setminus\bigcup_{e\in F}Z(\mu e).
\]
These sets form a basis of compact open subsets for a locally compact Hausdorff topology on $\partial E$.

The boundary path groupoid is
\[
 \G_E
 =\{(x,m-n,y)~|~x,y\in\partial E,\ \sigma^m(x)=\sigma^n(y)\},
\]
with the usual multiplication and inversion.  It is a Hausdorff ample groupoid with unit space $\partial E$.  The map
\[
 c_E:\G_E\longrightarrow\Z,
 \qquad
 c_E(x,k,y)=k,
\]
is a continuous cocycle.  There is a graded $*$-isomorphism
\begin{equation}\label{eq:lpa-steinberg}
 \Psi_E:L_{\K}(E)\stackrel{\cong}{\longrightarrow}A_{\K}(\G_E),
 \qquad
 \Psi_E(\mu\nu^*)=\one_{Z(\mu,\nu)},
\end{equation}
with the standard interpretation of generalized cylinders for arbitrary graphs (see \cite{ClarkFarthingSimsTomforde2014,Rigby2020}).

Every isotropy group of $\G_E$ is a subgroup of $\Z$ through the cocycle $c_E$.  It is nontrivial precisely at eventually periodic boundary paths.  Hence $\G_E$ has cyclic isotropy and $c_E$ is isotropy-faithful \cite{SteinbergVanWyk2026}.

\subsection{A boundary path rigidity lemma}

The compact open subsets of $\partial E$ need not be metrizable for an arbitrary graph.  Nevertheless, they satisfy a stronger sequential property tailored to \Cref{thm:local-baer-stonean}.

\begin{lemma}\label{lem:boundary-sequence}
Every nonisolated point of $\partial E$ is the limit of a nontrivial sequence of distinct boundary paths.  Consequently, every infinite compact open subset of $\partial E$ contains a nontrivial convergent sequence.
\end{lemma}

\begin{proof}
Let $x\in\partial E$ be nonisolated.

If $x=\mu$ is finite, then $\mathbf{r}(\mu)$ cannot be a sink, because $Z(\mu)=\{\mu\}$ in that case.  Thus $\mathbf{r}(\mu)$ is an infinite emitter.  Choose distinct edges $e_n\in \mathbf{s}^{-1}(\mathbf{r}(\mu))$ and choose a boundary path $y_n$ beginning at $\mathbf{r}(e_n)$.  Then
\[
 x_n=\mu e_ny_n
\]
converges to $\mu$: a basic neighborhood $Z(\mu\setminus F)$ excludes only finitely many first edges.

If $x=x_1x_2\cdots$ is infinite, then for each $n$ the cylinder $Z(x_1\cdots x_n)$ contains a point $y_n\neq x$.  We may choose $y_n$ so that it first differs from $x$ after the $n$th edge.  Then $y_n\to x$, and a subsequence is pairwise distinct.

Now let $U\subseteq\partial E$ be infinite and compact open.  If every point of $U$ were isolated, then $U$ would be a compact discrete space and hence finite.  Thus $U$ contains a nonisolated point.  Since $U$ is open, the sequence above can be chosen eventually inside $U$.
\end{proof}



\begin{proposition}\label{cor:graph-stonean-finite}
If every compact open subset of $\partial E$ is extremally disconnected, then every compact open subset of $\partial E$ is finite.  In particular, $\partial E$ and $\G_E$ are discrete.
\end{proposition}

\begin{proof}
	By the standard fact that a extremally disconnected Hausdorff space contains no nontrivial convergent sequences, an infinite compact open subset would contain a nontrivial convergent sequence by \Cref{lem:boundary-sequence}, a contradiction. 
 Thus compact open neighborhoods are finite, and $\partial E$ is discrete.  Since $\G_E$ is etale, discreteness of its unit space implies discreteness of the groupoid.
\end{proof}

\subsection{Discreteness in graph language}

Following \cite{HazratVas2018}, an infinite path is called an \emph{infinite sink} if it has neither cycles nor exits.  We say that an infinite path \emph{ends in a sink or a cycle} if, after deleting a finite initial segment, it is either an infinite sink or an infinite repetition of a cycle.

\begin{proposition}\label{prop:graph-discrete}
For a directed graph $E$, the following are equivalent.
\begin{enumerate}
\item\label{prop7.4_1} $\partial E$ is discrete.
\item\label{prop7.4_2} $\G_E$ is discrete.
\item\label{prop7.4_3} $E$ is row-finite and no-exit, and every infinite path ends in a sink or a cycle.
\end{enumerate}
\end{proposition}

\begin{proof}
The equivalence of \ref{prop7.4_1} and \ref{prop7.4_2} follows from etaleness, as in \Cref{cor:local-baer-discrete}.

Assume \ref{prop7.4_1}.  If $v$ were an infinite emitter, then the finite boundary path $v$ would be nonisolated by the first construction in \Cref{lem:boundary-sequence}.  Hence $E$ is row-finite.  If a cycle $c$ had an exit, the periodic path $c^\infty$ would be approached by paths that follow $c$ many times and then take the exit, so it would not be isolated.  Thus $E$ is no-exit.

Let $x=x_1x_2\cdots$ be infinite.  Since $x$ is isolated, a generalized cylinder $Z(\mu\setminus F)$ determined by a finite prefix contains only $x$.  Extend $\mu$ by the next edge of $x$.  From that point onward, every vertex along $x$ must emit only the edge used by $x$; otherwise one could branch and extend to another boundary path still lying in $Z(\mu\setminus F)$.  Thus a tail of $x$ has no exits.  If no vertex repeats, the tail is an infinite sink.  If a vertex repeats, the segment between two occurrences is a cycle and the no-exit condition forces the tail to repeat it forever.  This proves \ref{prop7.4_3}.

Conversely, assume \ref{prop7.4_3}.  A finite boundary path ends at a sink, because row-finiteness rules out infinite emitters, and hence is isolated.  If an infinite path eventually repeats a no-exit cycle, a cylinder beginning at the periodic tail is a singleton.  If it eventually becomes an infinite sink, a cylinder beginning where the exit-free tail starts is again a singleton.  Thus every boundary path is isolated.
\end{proof}



\subsection{The orbit decomposition of a no-exit Leavitt path algebra}

Two boundary paths lie in the same orbit of $\G_E$ precisely when they are shift-tail equivalent.  Let $\mathfrak A(E)$ be the set of aperiodic orbits and $\mathfrak P(E)$ the set of eventually periodic orbits.

\begin{corollary}\label{cor:graph-orbit-decomp}
Suppose $E$ satisfies the equivalent conditions of \Cref{prop:graph-discrete}.  Then
\begin{equation}\label{eq:graph-ungraded-decomp}
 L_{\K}(E)
 \cong
 \bigoplus_{\cO\in\mathfrak A(E)}M_{\cO}^{\fin}(\K)
 \;\oplus\!
 \bigoplus_{\cO\in\mathfrak P(E)}M_{\cO}^{\fin}(\K[t,t^{-1}])
\end{equation}
as algebras with involution.

If $p_{\cO}$ is the least positive period of an orbit $\cO\in\mathfrak P(E)$, then the natural $\Z$-grading corresponds to shifted matrix gradings over $\K$ and $\K[t^{p_{\cO}},t^{-p_{\cO}}]$.
\end{corollary}

\begin{proof}
By \Cref{prop:graph-discrete}, $\G_E$ is discrete.  Aperiodic units have trivial isotropy, while an eventually periodic unit has isotropy infinite cyclic and the cocycle maps a positive generator to its least period.  Apply \Cref{thm:orbit-matrix} and the isomorphism \eqref{eq:lpa-steinberg}.
\end{proof}

This is the groupoid origin of the familiar decomposition of no-exit Leavitt path algebras into finitary matrix algebras over $\K$ and $\K[t,t^{-1}]$.  Sinks and infinite sinks produce trivial-isotropy orbits; cycles produce Laurent polynomial blocks.  The cardinality of an orbit is the number of compatible finite approaches to the corresponding terminal object.

\subsection{Rickart and Baer characterizations of Leavitt path algebras}

We can now state the full application.  The local Baer $*$ statement below includes the row-finiteness and infinite-path hypotheses that are needed in the proof.  Their omission from the printed statement of \cite[Theorem 16]{HazratVas2018} is a typographical error noted explicitly in \cite[Remark 6.8]{JubeirVanWyk2024}.

\begin{theorem}	\label{thm:lpa-main}
Let $E$ be an arbitrary graph and $\K$ a field.

\begin{enumerate}
\item\label{theo7.7_1} \emph{Rickart conditions.}
The algebra $L_{\K}(E)$ is locally Rickart and hence graded locally Rickart.  The following are equivalent:
\begin{enumerate}[label=\textup{(\roman*)}]
\item $E^0$ is finite;
\item $L_{\K}(E)$ is right Rickart;
\item $L_{\K}(E)$ is left Rickart;
\item $L_{\K}(E)$ is graded right Rickart;
\item $L_{\K}(E)$ is graded left Rickart.
\end{enumerate}
If $\K$ is positive definite, then $L_{\K}(E)$ is graded locally Rickart $*$, and the conditions above are also equivalent to $L_{\K}(E)$ being graded Rickart~$*$.

\item\label{theo7.7_2}  \emph{Local Baer and graded local Baer conditions.}
The following are equivalent:
\begin{enumerate}[label=\textup{(\roman*)}]
\item $E$ is row-finite and no-exit, and every infinite path ends in a sink or a cycle;
\item $L_{\K}(E)$ is locally Baer;
\item $L_{\K}(E)$ is graded locally Baer.
\end{enumerate}
If $\K$ is positive definite, these are also equivalent to $L_{\K}(E)$ being graded locally Baer $*$.

\item\label{theo7.7_3}  \emph{Unital Baer and graded Baer conditions.}
The following are equivalent:
\begin{enumerate}[label=\textup{(\roman*)}]
\item $E$ is finite and no-exit;
\item $L_{\K}(E)$ is Baer;
\item $L_{\K}(E)$ is graded Baer.
\end{enumerate}
If $\K$ is positive definite, these are also equivalent to $L_{\K}(E)$ being graded Baer $*$.

\item\label{theo7.7_4}  \emph{Ungraded local Baer $*$ condition.}
Assume $\K$ is positive definite.  Then $L_{\K}(E)$ is locally Baer $*$ if and only if $E$ is a disjoint union of components of the following two types:
\begin{enumerate}[label=\textup{(\roman*)}]
\item a row-finite acyclic graph in which every infinite path ends in an infinite sink;
\item an isolated loop.
\end{enumerate}

\item\label{theo7.7_5}  \emph{Ungraded Baer $*$ condition.}
Assume $\K$ is positive definite.  Then $L_{\K}(E)$ is Baer $*$ if and only if $E$ is a finite disjoint union of finite acyclic graphs and isolated loops.
\end{enumerate}
\end{theorem}

\begin{proof}
\ref{theo7.7_1} Leavitt path algebras are left and right hereditary, hence left and right semihereditary \cite{AraGoodearl2012}.  They have local units given by finite sums of vertices, so \Cref{lem:standard-ring-facts}\ref{lem2.3_1} gives local Rickart.  Equivalently, one may transport this fact through \eqref{eq:lpa-steinberg}.  If $E^0$ is finite, the algebra is unital and therefore Rickart.  Conversely, every Rickart or graded Rickart ring is unital, and a Leavitt path algebra is unital exactly when $E^0$ is finite.  The involution interchanges right and left annihilators.

If $\K$ is positive definite, \Cref{thm:positive definite} shows that the Steinberg model is positive definite.  The natural graph grading is graded regular: this follows from Hazrat's theorem \cite{Hazrat2014} or from the neutral-component criterion of \cite{SteinbergVanWyk2026}.  Hence \Cref{cor:graded-rickart-star} gives graded local Rickart $*$.  The unital equivalence follows as above.

\ref{theo7.7_2} Suppose $L_{\K}(E)$ is locally Baer or graded locally Baer.  Through \eqref{eq:lpa-steinberg} and \Cref{thm:local-baer-stonean}, every compact open subset of $\partial E$ is extremally disconnected.  By \Cref{cor:graph-stonean-finite}, $\G_E$ is discrete.  \Cref{prop:graph-discrete} gives the graph condition.

Conversely, the graph condition makes $\G_E$ discrete.  Its isotropy is trivial or infinite cyclic, so the orbit decomposition and \Cref{lem:matrix-laurent}\ref{lem6.2_1} show that every finite local-unit corner is Baer.  This proves locally Baer, and hence graded locally Baer.  If $\K$ is positive definite, the natural cocycle is isotropy-faithful, and \Cref{lem:matrix-laurent}\ref{lem6.2_2} gives graded locally Baer $*$.

\ref{theo7.7_3} A Baer or graded Baer Leavitt path algebra is unital, so $E^0$ is finite.  Applying \ref{theo7.7_2}, the graph is row-finite and no-exit; with finitely many vertices this says that $E$ is finite and no-exit.  Conversely, for a finite no-exit graph, \ref{theo7.7_2} gives the local property and finite $E^0$ supplies an identity.  The graded Baer $*$ statement is identical.

\ref{theo7.7_4} First suppose $L_{\K}(E)$ is locally Baer $*$.  The compact open unit corners are Baer $*$ and hence Baer, so the argument in \ref{theo7.7_2} again shows that $\G_E$ is discrete.  By \Cref{thm:cyclic-local}\ref{theo6.3_3}, every orbit carrying nontrivial isotropy is a singleton.  Therefore every cycle is a loop: a cycle of length greater than one gives distinct cyclic shifts in the same periodic orbit.  No path from another vertex can enter such a loop, because adjoining that finite entrance to the periodic tail gives a second unit in the orbit.  Since discreteness already gives the no-exit condition, each cyclic component is an isolated loop.  Every remaining component is acyclic; \Cref{prop:graph-discrete} says it is row-finite and each infinite path ends in an infinite sink.

Conversely, an acyclic component has only trivial isotropy, and the stated path condition makes its boundary groupoid discrete.  An isolated loop contributes the one-point groupoid with isotropy $\Z$.  Thus every trivial-isotropy finite corner is a matrix algebra over $\K$, while every nontrivial-isotropy orbit contributes only $\K[t,t^{-1}]$.  These are Baer $*$ by \Cref{lem:matrix-laurent}\ref{lem6.2_3}, so \Cref{thm:discrete-local-criterion} gives local Baer $*$.

\ref{theo7.7_5} 
A Baer $*$-ring is unital, so combine \ref{theo7.7_4} with finiteness of $E^0$. 
Row-finite acyclic components with finitely many vertices are finite, and there are only finitely many components.  The converse follows by the same finite-product argument.
\end{proof}

\begin{corollary}\label{cor:lpa-positive}
For every nonempty graph $E$, the canonical involution on $L_{\K}(E)$ is positive definite if and only if the involution on $\K$ is positive definite.
\end{corollary}

\begin{proof}
Transport \Cref{thm:positive definite} through \eqref{eq:lpa-steinberg}.
\end{proof}

\begin{remark}\label{rem:why-star-stricter}
The difference between parts \ref{theo7.7_3} and \ref{theo7.7_5} of \Cref{thm:lpa-main} is not a mysterious graph artifact.  A no-exit cycle whose periodic orbit contains $n$ boundary paths contributes a Laurent block $M_n(\K[t,t^{-1}])$.  This block is Baer for every $n$, and it is graded Baer $*$ for the graph grading, but it is ungraded Rickart $*$ only when $n=1$.  The condition $n=1$ says exactly that the cycle is an isolated loop.
\end{remark}

\section{Examples}\label{sec:examples}

We close by giving examples that separate the different mechanisms in the paper.

\begin{example}[Pair groupoids]
Let $I$ be any set and let $\G=I\times I$ be the discrete pair groupoid.  It has trivial isotropy and one orbit, so
\[
 A_{\K}(\G)\cong M_I^{\fin}(\K).
\]
The algebra is locally Baer for every field and locally Baer $*$ for every positive definite $*$-field.  It is Baer precisely when $I$ is finite.  This example isolates orbit cardinality from isotropy.
\end{example}

\begin{example}[A cyclic transitive orbit]
Let $\G$ be a discrete transitive groupoid with unit orbit $I$ and isotropy $\Z$.  Then
\[
 A_{\K}(\G)\cong M_I^{\fin}(\K[t,t^{-1}]).
\]
It is locally Baer.  With an isotropy-faithful grading and positive definite coefficients it is graded locally Baer $*$.  For positive definite $\K$, it is locally Baer $*$ exactly when $|I|=1$.  This is the universal local model for the cycle obstruction in a Leavitt path algebra.
\end{example}

\begin{example}[A Cantor unit groupoid]
Let $X$ be the Cantor set, viewed as a unit groupoid.  Then $A_{\K}(X)$ is von Neumann regular: for a locally constant function $f$, define $g(x)=f(x)^{-1}$ on the support of $f$ and $g(x)=0$ elsewhere.  Thus $A_{\K}(X)$ is locally Rickart, and for positive definite $\K$ it is locally Rickart $*$.  It is not Baer, because $X$ is not extremally disconnected.  Hence regularity and even $*$-regularity do not imply the Baer condition. 
\end{example}

\begin{example}[A nondiscrete Baer Steinberg algebra]
As in \Cref{rem:metric-essential}, the unit groupoid on an infinite Stonean compact space $X$ has a Baer Steinberg algebra.  This shows that \Cref{cor:cyclic-unital} is not the final form of a classification for arbitrary nonmetrizable unit spaces.  
\end{example}

\begin{example}[Three graph blocks]
Let $\K$ be positive definite.
\begin{enumerate}
\item A finite acyclic graph gives a finite product of matrix algebras over $\K$ and therefore a Baer $*$ Leavitt path algebra.
\item A no-exit cycle of length $n>1$ gives $M_n(\K[t,t^{-1}])$, which is Baer and graded Baer $*$ but not Rickart $*$.
\item An isolated loop gives $\K[t,t^{-1}]$, which is Baer $*$.
\end{enumerate}
These are the atomic blocks behind parts \ref{theo7.7_3} and \ref{theo7.7_5} of \Cref{thm:lpa-main}.
\end{example}

\bibliographystyle{amsplain}
\bibliography{references}

\end{document}